\font\tenscr=rsfs10 
\font\sevenscr=rsfs7 
\font\fivescr=rsfs5 
\skewchar\tenscr='177 \skewchar\sevenscr='177 \skewchar\fivescr='177
\newfam\scrfam \textfont\scrfam=\tenscr \scriptfont\scrfam=\sevenscr
\scriptscriptfont\scrfam=\fivescr
\def\scr{\fam\scrfam}

\documentclass[12pt,reqno]{amsart}
\usepackage{graphicx}
\usepackage{amssymb}
\usepackage{amsmath}
\usepackage{amsthm}
\usepackage[latin1]{inputenc}
\usepackage{dsfont}
\usepackage{stmaryrd}
\usepackage{hyperref}
\usepackage{fullpage}

\newtheorem{theorem}{Theorem}[section]

\newtheorem{proposition}[theorem]{Proposition}

\theoremstyle{definition}

\theoremstyle{remark}
\newtheorem{remark}[theorem]{Remark}

\numberwithin{equation}{section}

\newcommand{\E}{\mathbb{E}}

\renewcommand{\O}{\mathcal{O}}

\newcommand{\R}{\mathbb{R}}
\renewcommand{\S}{\mathcal{S}}

\newcommand{\U}{\mathcal{U}}
\newcommand{\W}{\mathcal{W}}

\newcommand{\Z}{\mathbb{Z}}

\newcommand{\cd}{\,\cdot\,}

\begin{document}

\title[On a flow of tranformations of a Wiener space ]
{On a flow of transformations of a Wiener space}
\author{J. Najnudel, D. Stroock, M. Yor
}
\date{\today}
\maketitle
\begin{abstract}
In this paper, we define, via Fourier transform, an ergodic flow of transformations of a Wiener space which 
preserves the law of the Ornstein-Uhlenbeck process and which interpolates 
the iterations of a transformation previously defined by Jeulin and Yor. Then, we give a more 
explicit expression for this flow, and we construct from it a continuous gaussian 
process indexed by $\mathbb{R}^2$, such that all its restriction obtained by 
fixing the first coordinate are Ornstein-Uhlenbeck processes. 
\end{abstract}
\begin{section}{Introduction}\end{section}

An abstract Wiener space is a triple $(H,E,\mathcal W)$ consisting of a
separable, real Hilbert space $H$, a separable real Banach space $E$ in
which $H$ is continuously embedded as a dense subspace, and a Borel probability measure
$\mathcal W$ on $E$ with the property that, for each $x^*\in E^*$, the
$\W$-distribution of the map $x \in E\longmapsto
\langle x,x^*\rangle \in\R$, from $E$ to $\mathbb{R}$, is a
centered gaussian random distribution
with variance $\|h_{x^*}\|_H^2$, where $h_{x^*}$ is the element of $H$
determined by $(h,h_{x^*})_H=\langle h,x^*\rangle$ for all $h\in H$.  See
Chapter 8 of \cite{St1} for more information on this topic.

Because $\{h_{x^*}:\,x^*\in E^*\}$ is dense in $H$ and $\|h_{x^*}\|_H=
\|\langle\,\cdot\,,x^*\rangle\|_{L^2(\mathcal W)}$, there is a unique
isometry, known as the Paley--Wiener map,
$\mathcal I:H\longmapsto L^2(\mathcal W)$ such that $\mathcal
I(h)=\langle \,\cdot\,,x^*\rangle$ if $h=h_{x^*}$.  In fact, for each
$h\in H$, $\mathcal I(h)$ under $\mathcal W$ is a centered Gaussian variable with
variance $\|h\|_H^2$.  Because, when $h=h_{x^*}$, $\mathcal I(h)$ provides an
extention of $(\,\cdot\,,h)_H$ to $E$, for intuitive purposes one can think
of $x\rightsquigarrow[\mathcal I(h)](x)$ as a giving meaning to the inner
product $x\rightsquigarrow (x,h)_H$, although for general $h$ this will be
defined only up to a set of $\mathcal W$-measure $0$.

An important property of abstract Wiener spaces is that they are invariant
under orthogonal transformations on $H$.  To be precise, given an
orthogonal transformation $\mathcal O$ on $H$, there is a $\mathcal
W$-almost surely unique $T_{\mathcal O}:E\longrightarrow E$ with the
property that, for each $h\in H$, $\mathcal I(h)\circ T_{\mathcal
  O}=\mathcal I(\mathcal O^\top 
h)$ $\mathcal W$-almost surely.  Notice that this is the relation which
one would predict if one thinks of $[\mathcal I(h)](x)$ as the inner
product of $x$ with $h$.  In general, $T_\O$ can be constructed by choosing
$\{x_m^*:\,m\ge1\}\subseteq E^*$ so the $\{h_{x^*_m}:\,m\ge1\}$ is an
orthonormal basis in $H$ and then taking
$$T_\O x=\sum_{m=1}^\infty \langle x,x_m^*\rangle \O h_{x_m^*},$$
where the series converges in $E$ for $\W$-almost every $x$ as well as in
$L^p(\W;E)$ for every $p\in[1,\infty )$.  See Theorem 8.3.14 in \cite{St1}
  for details.  In the case when $\O$ admits an
extension as a continuous map on $E$ into itself, $T_{\mathcal O}$ can
be the taken equal to that extension.
In any case, it is an easy matter to check that the measure 
$\mathcal{W}$ is preserved by $T_{\mathcal O}$.  Less obvious is a theorem, originally
formulated by I.M. Segal 
(cf.\ \cite{St2}), which says that $T_{\mathcal O}$ is ergodic if and only
$\mathcal O$ admits no non-trivial, finite dimensional, invariant
subspace.  Equivalently, $T_{\mathcal O}$ is ergodic if and only if the
complexification $\mathcal O_{\rm c}$ has a continuous spectrum as a
unitary operator on the complexification $H_{\rm c}$ of $H$.

The classical Wiener space provides a rich source of examples to which the
preceding applies.  Namely, take $H=H^1_0$ to be
  the space of absolutely continuous $h\in\Theta $ whose derivative $\dot
  h$ is in $L^2([0,\infty ))$, and set $\|h\|_{H^1_0}=\|\dot
    h\|_{L^2([0,\infty ))}$.  Then $H^1_0$ with norm
      $\|\,\cdot\,\|_{H^1_0}$ is a separable Hilbert space.  
Next, take $E=\Theta $, where $\Theta $ is the space
of continuous paths $\theta :[0,\infty )\longrightarrow \mathbb R$ such that
$\theta (0)=0$ and
$$\frac{|\theta (t)|}{t^{\frac12}\log(e+|\log t|)}\longrightarrow
    0\quad\text{as $t > 0$ tends to $0$ or $\infty$},$$
and set
$$\|\theta \|_\Theta  =\sup_{t>0}\frac{|\theta
  (t)|}{t^{\frac12}\log(e+|\log t|)}.$$
Then $\Theta $ with norm $\|\,\cdot\,\|_\Theta $ is a separable Banach
space in which $H^1_0$ is continuously embedded as a dense subspace.  
Finally, the renowned theorem of Wiener combined with the Brownian law of
the iterated logarithm says that there is a Borel
probability measure $\mathcal W_{H^1_0}$ on $\Theta $ for which
$(H^1_0,\Theta ,\mathcal W_{H^1_0})$ is an abstract Wiener space.  Indeed,
it is the classical Wiener space on which the abstraction is modeled, and
$\mathcal W_{H^1_0}$ is the distribution of an $\mathbb R$-valued Brownian motion.

One of the simplest examples of an orthogonal transformation on $H^1_0$ for
which the associated transformation on $\Theta $ is ergodic is the
Brownian scaling map $S_\alpha $ given by $S_\alpha \theta (t)=\alpha
^{-\frac12}\theta (\alpha t)$ for $\alpha >0$.  It is an easy matter to
check that the restriction ${\mathcal O}_\alpha $ of $S_\alpha $ to $H^1_0$
is orthogonal, and so, since $S_\alpha $ is
continuous on $\Theta $, we can take $T_{\mathcal O_\alpha }=S_\alpha $.
Furthermore, as long as $\alpha \neq1$, an elementary computation 
shows that $\lim_{n\to\infty }\bigl(g,\mathcal O^n_\alpha h\bigr)_H=0$,
first for smooth $g,\,h\in H^1_0$ with compact support in $(0,\infty )$ and
thence for all $g,\,h\in H^1_0$.  Hence, when $\alpha \neq1$, $\mathcal
O_\alpha $ admits no 
non-trivial, finite dimensional subspace, and therefore $S_\alpha $ is
ergodic; and so, by the Birkoff's Individual Ergodic Theorem, for
$p\in[1,\infty )$ and $f\in
L^p(\mathcal W_{H^1_0})$,
$$\lim_{n\to\infty }\frac1n\sum_{m=0}^{n-1}f\circ S_\alpha ^n=\int
f\,d\mathcal W_{H^1_0}$$
both $\mathcal W_{H^1_0}$-almost surely and in $L^p(\mathcal W_{H^1_0})$.
Moreover, since $\{S_\alpha :\,\alpha \in(0,\infty )\}$ is a
multiplicative semigroup in the sense that $S_{\alpha \beta }=S_\alpha
\circ S_\beta $, one has the continuous parameter version
$$\lim_{a\to\infty }\frac1{\log a}\int_1^a (f\circ S_\alpha) \,\frac{d\alpha}{\alpha} =\int
f\,d\mathcal W_{H^1_0}$$
of the preceding result.

A more challenging ergodic transformation of the classical Wiener space was
studied by Jeulin and Yor (see \cite{JY}, \cite{M} and \cite{Y}),  
and, in the framework of this article, it is obtained by
considering the transformation $\mathcal
O$ on $H^1_0$, defined by
\begin{equation}\label{JY}[\mathcal Oh](t)=h(t)-\int_0^t\frac{h(s)}s\,ds.\end{equation}
An elementary calculation shows that $\mathcal O$ is orthogonal.  Moreover,
$\mathcal O$ admits a continuous extension to $\Theta $ given by replacing
$h\in H^1_0$ in (\ref{JY}) by $\theta \in\Theta $.  That is
\begin{equation}\label{EJY}
[T_{\mathcal O}\theta ]=\theta (t)-\int_0^t\frac{\theta (s)}s\,ds\quad\text{for
  $\theta\in\Theta$ and $t\ge0$}.\end{equation} 
In addition, one can check that $\lim_{n\to\infty
}\bigl(g,\mathcal O^nh\bigr)_{H^1_0}=0$ for all $g,\,h\in H^1_0$, which
proves that $T_{\mathcal O}$ is ergodic  for $\mathcal W_{H^1_0}$.   

In order to study the transformation $T_{\mathcal O}$ in greater detail, it
will be convenient to reformulate it in terms of the Ornstein--Uhlenbeck
process.  That is, take
$H^U$ to be the space of absolutely continuous functions
$h:\R\longrightarrow \R$ such that
$$\|h\|_{H^U}\equiv\sqrt{\int_{\mathbb R}\bigl(\tfrac14 h(t)^2+\dot
  h(t)^2\bigr)\,dt}<\infty .$$
Then $H^U$ becomes a separable Hilbert space with norm $\|\cd\|_{H^U}$.
Moreover, the map $F:H^1_0\longrightarrow H^U$ given by
\begin{equation}\label{F}[F(g)](t)=e^{-\frac t2}g(e^t),\quad\text{for }
  g\in H^1_0\text{ and }t\in\R,\end{equation}
is an isometric surjection which extends as an isometry from $\Theta $ onto Banach
space $\mathcal U$ of continuous $\omega:\longrightarrow \R$ satisfying $\lim_{|t|\to\infty
}\frac{|\omega(t)|}{\log |t|}=0$ with norm $\|\omega\|_{\mathcal U}=\sup_{t\in\mathbb
  R}\bigl(\log(e+|t|)\bigr)^{-1}|\omega(t)|$.  Thus, 
$(H^U,\mathcal U,\mathcal W_{H^U})$ is an abstract Wiener space, where 
$\mathcal W_{H^U}=F_*\mathcal W_{H^1_0}$ is the image of $\mathcal W_{H^1_0}$ under the map 
$F$.  In fact,  $\mathcal W_{H^U}$ is the distribution of 
a standard, reverisible Ornstein--Uhlenbeck process.

Note that the scaling transformations  for the classical Wiener space
become translations in the Ornstein--Uhlenbeck setting.  Namely,
for each $\alpha > 0$, $F\circ S_\alpha=\tau_{\log\alpha }\circ F$,
where 
$\tau _s$ denotes the time-translation map given by $[\tau _s \omega](t)=\omega(s+t)$.
Thus, for $s\neq 0$, the results proved about the scaling maps say that
$\tau _s$ is an ergodic transformation for $\mathcal W_{H^U}$.  In
particular, for $p\in [1,\infty )$ and $f\in L^p(\mathcal W_{H^U})$,
$$\lim_{n\to\infty }\frac1n\sum_{m=0}^{n-1}f\circ\tau _{ns}=\lim_{T\to\infty }\frac1T
\int_0^T f\circ\tau _s\,ds=\int f\,d\mathcal W_{H^U}$$
both $\mathcal W_{H^U}$-almost surely and in $L^p(\mathcal W_{H^U})$.

The main goal of this article is to show that the reformulation of transformation
$T_{\mathcal   O}$ coming from the Jeulin--Yor transformation
in terms of the Ornstein--Uhlenbeck process allows us to embed $T_\O$ in a
continuous-time flow of transformations on the space $\mathcal{U}$, each of
which is $\mathcal W_{H^1_0}$-measure preserving and all but one of which
is ergodic. In Section \ref{2}, this flow
is described via Fourier transforms. In Section
\ref{3}, a direct and more explicit expression, involving
hypergeometric functions and principal values, is computed. In Section
\ref{4}, we study the two-parameter gaussian process which is induced by the flow
introduced in Section \ref{2}. In particular, we compute its covariance and
prove that it admits a version which is jointly continuous in its
parameters.  

\medbreak
\begin{section}{Preliminary description of the flow} \label{2} \end{section}

Let $\mathcal O$ and $T_\O$ be the transformations on $H^1_0$ and $\Theta $
given by (\ref{JY}) and (\ref{EJY}), and recall the unitary map
$F:H^1_0\longrightarrow H^U$ in (\ref{F}) and its continuous extension as
an isometry from $\Theta $ onto $\U$.  Clearly, the inverse of $F$ is given by 
$$F^{-1}(\omega) (t) = \sqrt{t} \, \omega(\log t)\quad\text{ for }t>0.$$

Because $F$ is unitary and $\O$ is orthogonal on $H^1_0$, $-F\circ \O\circ
F^{-1}$ is an orthognal transformation on $H^U$, and because
$$S:= - F \circ T_{\mathcal O} \circ F^{-1}$$
is continuous extension of $-F\circ \O\circ F^{-1}$ to $\U$, we can
identify $S$ as $T_{-F\circ \O\circ F^{-1}}$.  

Another expression for action of $S$ is
$$[S(\omega )](t)=-\omega (t)+\int_0^\infty e^{-\frac s2}\omega
(t-s)\,ds\quad\text{for } t\in\R.$$
Equivalently, 
$$S(\omega) = \omega * \mu,$$
where $\mu $ is the finite, signed measure $\mu$ given by
$$\mu := - \delta_0 + e^{-\frac t2} \mathds{1}_{t \geq 0} dt.$$
To confirm that $\omega *\mu $ is well-defined as a Lebesgue integral and
that it maps $\U$ continuously into itself, 
note that, for any $\omega \in\U$ and $t\in\R$,
\begin{align*}\int_0^\infty e^{-\frac s2}|\omega (t-s)|\,ds&\le\|\omega \|_\U
\int_0^\infty e^{-\frac s2}\log\bigl(e+|t|+s)\,ds\\&\le
\|\omega \|_\U\log(e+|t|)\int_0^\infty e^{-\frac s2}(1+s)\,ds\le 9\|\omega \|_\U\log(e+|t|) 
\end{align*}

The Fourier transform $\widehat \mu $ of $\mu $ is given by
$$ \widehat{\mu} (\lambda)  = \int_{\mathbb{R}} e^{-i \lambda x} d\mu(x) 
 = - 1 + \int_0^{\infty} e^{- x (1/2 + i \lambda)} dx 
 = - 1 + \frac{1}{1/2 + i \lambda} = \frac{1 - 2i \lambda}{1 + 2 i \lambda}
= e^{- 2 i \operatorname{Arctg} (2 \lambda)}. $$
Hence, for all $h \in H^U$ and $\lambda \in \mathbb{R}$, 
\begin{equation}
\label{FS}  
\widehat{h * \mu} (\lambda)= e^{- 2 i \operatorname{Arctg} (2
  \lambda)}\widehat{h}(\lambda),\end{equation} 
which, since
$$\|h\|_{H^U}^2=\frac1{8\pi }\int_\R|\widehat h(\lambda )|^2\bigl(1+4\lambda
^2\bigr)\,d\lambda ,$$
provides another proof that $S\restriction H^U$ is isometric.  

The preceeding, and especially (\ref{FS}), suggests a natural way to embed
$S\restriction H^U$ into a continuous group of orthogonal
transformations.  Namely, for $u\in\R$, let
$\mu ^{*u}$ to be the unique tempered distribution whose Fourier transform
is given by
\begin{equation}\label{alpha}\widehat{\mu^{*u}} (\lambda) = e^{- 2 i u
  \operatorname{Arctg} (2 \lambda)},\end{equation}
and define $\S^u\varphi =\varphi
*\mu ^{*u}$ for $\varphi $ in the Schwartz test function class $\scr S$ of smooth
functions which, together with all their derivatives, are rapidly
decreasing.  Because $$\widehat{\S^u\varphi }(\lambda )=e^{- 2 i u
  \operatorname{Arctg} (2 \lambda)}\hat\varphi (\lambda ),$$ it is obvious
that $\S^u$ has a unique extension as an orthogonal transformation on
$H^U$, which we will again denote by $\S^u$.  Furthermore, it is clear that
$\S^{u+v}=\S^u\circ \S^v$ for all $u,v\in\R$.  Finally, for all $g,h \in
H^U$, $u \in \mathbb{R}$,
\begin{align*} (g, \mathcal{S}^u h)_{H^U} & =
  \frac{1}{8\pi} \int_{\mathbb{R}} \overline{\widehat{g}(\lambda)} \,
  \widehat{h}(\lambda) \, e^{- 2 i u \operatorname{Arctg} (2 \lambda)} (1 +
  4 \lambda^2) \, d\lambda \\ & = \frac{1}{16 \pi} \int_{- \pi/2}^{\pi/2}
  \overline{\widehat{g}\left( \frac{\tan(\tau)}{2} \right)} \, \widehat{h}
  \left(\frac{\tan(\tau)}{2} \right)\, \left(1 + \tan^2(\tau) \right)^2
  e^{- 2 i u \tau} \, d\tau, \end{align*} \noindent

where
$$ \frac{1}{16 \pi} \int_{- \pi/2}^{\pi/2} \left| \widehat{g}\left(
\frac{\tan(\tau)}{2} \right) \right| \, \left| \widehat{h}
\left(\frac{\tan(\tau)}{2} \right) \right| \left(1 + \tan^2(\tau) \right)^2
\, d\tau = \frac{1}{8\pi} \int_{\mathbb{R}} | \widehat{g}(\lambda) | \, |
\widehat{h}(\lambda) | \, (1 + 4 \lambda^2) \, d\lambda$$ $$ \leq
\frac{1}{8\pi} \, \left( \int_{\mathbb{R}} | \widehat{g}(\lambda) |^2 (1 +
4 \lambda^2) d \lambda \right)^{1/2} \, \left( \int_{\mathbb{R}} |
\widehat{h}(\lambda) |^2 (1 + 4 \lambda^2) d \lambda \right)^{1/2}=
||g||_{H^U} ||h||_{H^U} < \infty.$$
Hence, by Riemann--Lebesgue lemma, shows that $(g, \mathcal{S}^u h)_{H^U}$ tends to
zero when $|u|$ goes to infinity.

Now define the associated transformations $S^u:=T_{\S^u}$ on $\U$ for each $u\in\R$.
By the general theory summarized in the introduction and the
preceding discussion, we know that $\{S^u:\,u\in\R\}$ is
a flow of $\W_{H^U}$-measure preserving transformations and that, for each
$u\neq0$, $S^u$ is ergodic.  

\medbreak
\begin{section}{A more explicit expression} \label{3} \end{section}

So far we know very little about the transformations $S^u$ for general
$u\in\R$.  By getting a handle on the tempered distributions $\mu ^{*u}$,
in this section we will attempt to find out a little more.

We begin with the case when $u$ is an integer $n\in\Z$.  Recalling that
$\mu =-\delta _0+e^{-\frac t2}{\mathds 1}_{t\ge 0}\,dt$, one can use
induction to check that, for $n\ge0$,
$$\mu^{*n} = (-1)^n \bigl(  \delta_0 + e^{-\frac t2} L'_n(t) \mathds{1}_{t
  \geq 0} dt\bigr),$$
where $L_n$ is the $n$th Laguerre polynomial. Indeed, the 
Laguerre polynomials satisfy the following relations: for all $n \geq 0$,
$$L_n(0) = 1$$  and for all $n \geq 0$, $t \in \mathbb{R}$,   
$$L'_{n+1}(t) = L'_n(t) - L_n(t).$$
 Similarly, starting from
$\mu ^{*-1}=-\delta _0+e^{\frac t2}{\mathds1}_{t\ge0}\,dt$, one finds that
$$\mu^{*n} = (-1)^n \bigl(  \delta_0 + e^{\frac t2} L'_n(-t) \mathds{1}_{t
  \leq 0} dt\bigr)$$ 
for $n\le0$.  In particular, $\mu ^{*n}$ is a finite, signed measure 
for $n\in{\Z}$ and $S^n\omega $ can be identified as $\mu ^{*n}*\omega $ for
all $\omega \in\U$ and $n\in\Z$.

As the next result shows, when $u\notin\Z$, $\mu ^{*u}$ is more singular
tempered distribution than a finite, signed measure.  
\newline
\begin{proposition}  \label{propo}
For each $u \notin \mathbb{Z}$, the distribution $\mu^{*u}$ is given by the
following formula: 
\begin{equation}
\mu^{*u} = \cos(\pi u) \delta_0(x) + \frac{\sin(\pi u)}{\pi} pv(1/x) +
\Phi_u(x), \label{y6} \end{equation}
where $pv$ denotes the principal value, and $\Phi_u\in L^2(\R)$ is 
 the function for which $\Phi _u(x)$ equals 
\begin{equation}\begin{aligned} e^{-|x|/2}  & \left( - \frac{u \, \sin(\pi
    u)}{\pi} \sum_{k = 0}^{\infty} 
\frac{(1-u \operatorname{sgn} (x))_k |x|^k}{k!(k+1)!} \left[
\frac{\Gamma'}{\Gamma} (1+k - u \operatorname{sgn} (x)) -
\frac{\Gamma'}{\Gamma} (1+k) \right. \right. \nonumber \\  
&\hskip2.7truein \left. \left. \;  \;  - 
\frac{\Gamma'}{\Gamma}(2+k) + \log (|x|) \right] + \frac{\sin(\pi u)}{\pi
  x} \right) - \frac{\sin \pi u}{\pi x}, \label{y7} \end{aligned}\end{equation}
 $\Gamma'/\Gamma$ being the logarithmic derivative of the Euler gamma function
and $(\; \;)_k$ being the Pochhammer symbol. 
\noindent
\end{proposition}
\begin{proof}
Define the functions 
$\psi_u$ and $\theta_u$ from $\mathbb{R}^*=\R\setminus \{0\}$ to
$\mathbb{R}$ so that $\theta _u(x)=e^{-\frac x2}\psi _u(x)$ and $\psi
_u(x)$ equals
\begin{align*}
- \frac{u \, \sin(\pi u)}{\pi} \sum_{k = 0}^{\infty}
\frac{(1-u \operatorname{sgn} (x))_k |x|^k}{k!(k+1)!} &\left[
\frac{\Gamma'}{\Gamma} (1+k - u \operatorname{sgn} (x)) -
\frac{\Gamma'}{\Gamma} (1+k) \right. \\  
 &\left. \hskip10pt  - 
\frac{\Gamma'}{\Gamma}(2+k) + \log (|x|) \right] + \frac{\sin(\pi u)}{\pi x}.
\end{align*}
\noindent
From Lebedev \cite{L}, p. 264, equation (9.10.6), 
with the parameters $\alpha = 1-u$ or $\alpha = 1+u$, $n =1$, $z = x$ or $z
= -x$, the function 
$\psi_u$ satisfies, for all $x \in \mathbb{R}^*$, the differential equation:
$$x \psi_u'' (x) + (2 - |x|) \psi'_u(x) + (u - \operatorname{sgn} (x)) \psi_u(x) = 0,$$
and grows at most polynomially at infinity. One then deduces that
$\theta_u$ decreases as least  
exponentially at infinity, and satisfies (for $x \neq 0$) the following equation:
\begin{equation}
x \theta''_u(x) + 2 \theta'_u(x) + \left( u - \frac{x}{4} \right) \,
\theta_u(x) = 0. \label{ed} \end{equation}
At the same time, by writing 
$$e^{-|x|/2} = (e^{-|x|/2} -1) + 1$$
and expanding $\theta_u(x)$ accordingly, we obtain:
\begin{align*}
\theta_u (x) =  &\frac{\sin(\pi u)}{\pi x} - 
\frac{ u \sin(\pi u)}{\pi} \left[ \frac{\Gamma'}{\Gamma} (1- u \operatorname{sgn} (x)) -
 \frac{\Gamma'}{\Gamma} (1)   - 
\frac{\Gamma'}{\Gamma}(2) + \log (|x|) \right] \\ & \; - \frac{\sin(\pi
  u)}{2 \pi} \operatorname{sgn} (x)  + \eta_u (x),\end{align*}
for
$$\eta_u (x) = x \eta_u^{(1)} (x) + |x| \eta_u^{(2)} (x) + x \log (|x|)
\eta_u^{(3)} (x) + |x| \log(|x|) \eta_u^{(4)} (x),$$ 
where $\eta_u^{(1)}$, $\eta_u^{(2)}$, $\eta_u^{(3)}$, $\eta_u^{(4)}$ are
all smooth functions. 
The derivatives of the functions $x$, $|x|$, $x \log |x|$, $|x| \log |x|$
in the sense of the distributions are obtained by interpreting their
ordinary derivatives as distributions. Similarly, the product by $x$
of their second distributional derivatives are obtained by multiplying 
their ordinary second derivatives by $x$.  Hence, both 
$\eta_u'(x)$ and $x \eta_u''(x)$ as distributions can be obtained by 
computing $\eta_u'(x)$ and  $x \eta_u''(x)$ as functions on $\R^*$.

Now, let $\nu_u$ be the distribution given by the expression:
\begin{equation}
\nu_u (x) = \cos (\pi u) \delta_0 (x) + \frac{\sin(\pi u)}{\pi} pv(1/x) +
\left[\theta_u(x) - \frac{\sin(\pi u)}{\pi x} \right]. \label{w37} 
\end{equation}
Note that the term in brackets, in the definition of $\nu_u$, is a locally
integrable function, and that  
 $\nu_u$ coincides with the function $\theta_u$ in the complement of the
neighborhood of zero.  
Let us now prove that $\nu_u$ satisfies the analog of the equation
\eqref{ed}, in the sense of the distributions.  
One has:
\begin{align*}
\nu_u(x) & = \cos(\pi u) \delta_0(x) + \frac{\sin(\pi u)}{\pi} pv(1/x) -
\frac{ u \sin(\pi u)}{\pi} 
\left[ \frac{\Gamma'}{\Gamma} (1 - u  \operatorname{sgn}(x)) \right. \\ & \; \left. 
- \frac{\Gamma'}{\Gamma} (1) - \frac{\Gamma'}{\Gamma}(2) + \log (|x|)
\right] - \frac{\sin(\pi u)}{2 \pi} 
\operatorname{sgn} (x)  + \eta_u (x). 
\end{align*} 
\noindent
Since 
$$\frac{\Gamma'}{\Gamma} (1+u) - \frac{\Gamma'}{\Gamma} (1-u) = 
\frac{\frac{d}{du} \left( \Gamma(1+u) \Gamma(1-u) \right)}{\Gamma(1+u)
  \Gamma(1-u)}  = \frac{ \frac{d}{du} (\pi u/ \sin (\pi u))}{\pi u /
  \sin(\pi u)} =  
\frac{1}{u} - \pi \cot (\pi u),$$
one obtains, after straightforward computation, 
$$\nu_u(x) = \cos(\pi) \delta_0 (x) + \frac{\sin(\pi u)}{\pi} pv(1/x) -
\frac{u \cos(\pi u)}{2} \operatorname{sgn} (x)  
- \frac{u \sin(\pi u)}{\pi} \log(|x|) + c(u) + \eta_u(x),$$
where $c(u)$ does not depend on $x$. 
One deduces that 
$$\nu_u(x) = \cos(\pi u) \delta_0(x) + \frac{\sin(\pi u)}{\pi} pv(1/x) + \chi_{u,1} (x),$$
where $\chi_{u,1}$ denotes a locally integrable function. 
Moreover, 
$$\nu'_u(x) = \cos (\pi u) \delta'_0(x) - \frac{\sin(\pi u)}{\pi} fp(1/x^2) -
u \cos(\pi u) \delta_0 (x) - \frac{u \sin (\pi u)}{\pi}  pv(1/x) + \eta'_u(x),$$
where $fp(1/x^2)$ denotes the finite part of $1/x^2$, and then 
$$x \nu'_u(x) = - \cos(\pi u) \delta_0(x) - \frac{\sin(\pi u)}{\pi} pv(1/x) 
-  \frac{u \sin (\pi u)}{\pi}  + x \eta'_u (x).$$
By differentiating again, one obtains:
$$\nu_u'(x) + x \nu''_u(x) = 
- \cos(\pi u) \delta_0'(x) + \frac{\sin (\pi u)}{\pi} fp (1/x^2) 
+ \eta'_u (x)  + x \eta''_u(x).$$
Therefore, 
$$
x \nu''_u(x) + 2 \nu'_u (x) +\left(u - \frac{x}{4} \right) \nu_u(x) 
 = \chi_{u,2} (x) + \left( - \cos (\pi u) \delta'_0 (x) + \frac{\sin(\pi u)}{\pi}
fp(1/x^2) \right) $$ $$+ \left( \cos(\pi u) \delta'_0 (x) - \frac{\sin (\pi u)}{\pi} 
fp(1/x^2) - u \cos(\pi u) \delta_0 (x) - \frac{ u \sin (\pi u)}{\pi} pv (1/x) \right)
$$ $$ + u \left(\cos(\pi u) \delta_0 (x) + \frac{ \sin (\pi u)}{\pi} 
pv(1/x) \right) = \chi_{u,2} (x),$$
\noindent
where $\chi_{u,2}$ is a locally integrable function. Since 
$\theta_u$ satisfies \eqref{ed}, $\chi_{u,2}$ is identically zero. 
Hence, $\nu_u$ is a tempered distribution solving the differential equation:
$$x \nu''_u(x) + 2 \nu'_u (x) + \left(u - \frac{x}{4} \right) \nu_u(x)  = 0,$$
or equivalently, 
$$\frac{x}{4} \nu_u(x) - \frac{d^2}{d^2 x}(x \nu_u (x)) - u \nu_u(x) = 0.$$
Multiplying  by $-4i$ and taking the Fourier transform (in the sense of the 
distributions), one deduces:
$$\widehat{\nu_u}' (\lambda) (1+ 4\lambda^2)  = - 4iu \widehat{\nu_u} (\lambda).$$
This linear equation admits a unique solution, up to a multiplicative factor 
$c$:
$$\widehat{\nu_u}(\lambda) = c \exp \left( \int_{0}^{\lambda} \frac{-4iu}{1
  + 4t^2} dt \right) 
= c \exp(- 2 i u  \operatorname{Arctg} (2 \lambda)).$$
Hence, $\nu_u$ is proportional to $\mu^{*u}$. In order to determine the constant 
$c$, let us observe that the distribution $\nu_{u,0}$ given by
$$\nu_{u,0}(x) = \nu_u(x) - c \cos(\pi u) \delta_0(x) - \frac{c \sin(\pi
  u)}{\pi} pv(1/x) $$ 
admits the Fourier transform:
$$\widehat{\nu_{u,0}}(\lambda) = c \, e^{- 2 i u  \operatorname{Arctg} (2
  \lambda)} - c \, e^{-\pi i u \operatorname{sgn} (\lambda)}.$$ 
One deduces that $\widehat{\nu_{u,0}}$ is a function in $L^2$, which
implies that $\nu_{u,0}$ is also a function in $L^2$, and  
then locally integrable. Since the last term in \eqref{w37} is also a
locally integrable function, one deduces that $c =1$,  
and then 
$$\mu^{*u} = \nu_u,$$
which proves Proposition \ref{propo}.
\end{proof}
\noindent

The reasonably explicit expression for $\mu ^{*u}$ found in Proposition
\ref{propo} yields a
reaonably explicit expression for the action of $\mathcal{S}^u$.  Indeed,
only the term $pv (1/x)$ is a source of concern.  However, convolution
with respect of $pv(1/x)$ is, apart from a multiplicative constant, just
the Hilbert transform, whose properties are well-known.  In particular, it
is a translation invariant, bounded map on $L^2(\R)$, and as such it is also a
bounded map on $H^U$.  Thus, we can unambiguously write
$\mathcal{S}^{u} (h) = h * \mu^{*u}$ for all $h \in H^U$.  On the other
hand, the interpretation of $\omega *\mu ^{*u}$ for $\omega \in\U$ needs
some thought.  No doubt, $\omega *\mu ^{*u}$ is well-defined as an element
of ${\scr S\,}'$, the space tempered
distributions, but it is not immediately obvious that it is can be
represented by an element of $\U$ or, if it can, that the element of
$\U$ which represents it can be identified as $S^u\omega $.  In fact, the
best that we should expect is that such statements will be true of
$\W_{H^U}$-almost every $\omega \in\U$.  The following result justifies that
expectation.

\begin{proposition} \label{OU}  
For $\mathcal W_{H^U}$-almost every $\omega \in \mathcal{U}$,
the tempered distribution $\omega * \mu^{*u}$ is represented by an element
of $\U$ which can be can be identified as $S^u\omega $.
\end{proposition}

\begin{proof}  Recall that, for $\varphi \in\scr S$, $\varphi *\mu ^{*-u}$ is
  the element of $\scr S$ whose Fourier transform is given by
$$
  \widehat{\varphi * \mu^{*-u}}(\lambda) = \widehat{\varphi}(\lambda) e^{2iu
    \operatorname{Arctg} (2 \lambda)}\quad\text{for all } \lambda \in\R.
$$
Also, if $T\in{\scr S\,}'$, then $T*\mu ^{*u}$ is the tempered distribution
whose action on $\varphi \in\scr S$ is given by
$${}_{\scr S}\langle \varphi ,T*\mu ^{*u}\rangle_{\scr S'}=
{}_{\scr S}\langle \varphi *\mu ^{*-u},T\rangle_{\scr S'}.$$

Now choose an orthonormal basis $\{h_n:\,n\ge1\}$ for $H^U$ all of whose
members are elements of $\scr S$, and, for each $n\ge1$, set $g_n=\frac14
h_n+h_n''$.  Next, think of $g_n$ as the element of $\U^*$ whose action on
$\omega \in\U$ is given by
$${}_\U\langle\omega ,g_n\rangle_{\U^*}={}_{\scr S}\langle g_n,\omega
\rangle_{\scr S'}.$$
It is then an easy matter to check that, in the notation of the
introduction, $h_n=h_{g_n}$.  Hence, if $B$ is the subset of $\omega \in
\U$ for which
$$\omega =\lim_{n\to\infty }\sum_{m=1}^n{}_{\scr S}\langle g_n,\omega
\rangle_{\scr S'}h_n\quad\text{and} \quad
S^u\omega =\lim_{n\to\infty }\sum_{m=1}^n{}_{\scr S}\langle g_n,\omega
\rangle_{\scr S'}h_n*\mu ^{*u},$$
where the convergence is in $\U$, then $\W_{H^U}(B)=1$.

Now let $\omega \in B$.  Then, for each $\varphi \in\scr S$,
\begin{align*}
{}_{\scr S}\langle \varphi ,\omega *\mu ^{*u}\rangle_{\scr S'}&=
{}_{\scr S}\langle\varphi *\mu ^{*-u},\omega \rangle_{\scr S'}=
\lim_{n\to\infty }\sum_{m=1}^n{}_{\scr S}\langle g_n,\omega \rangle_{\scr S'}
{}_{\scr S}\langle \varphi ,h_n*\mu ^{*u}\rangle_{\scr S'}\\&=
\lim_{n\to\infty }\sum_{m=1}^n{}_{\scr S}\langle g_n,\omega \rangle_{\scr S'}
{}_{\scr S}\langle \varphi ,\S^uh_n\rangle_{\scr S'}={}_{\scr S}\langle
\varphi ,S^u\omega \rangle_{\scr S'}.\end{align*}
Thus, for $\omega \in B$, $\omega *\mu ^{*u}\in{\scr S\,}'$ is represented
by $S^u\omega \in\U$. 
\end{proof}

\medbreak
\begin{section}{A two parameter gaussian process} \label{4} \end{section}

By construction, $\{S^u\omega
(t):\,(u,t)\in\R^2\}$ is a gaussian family in $L^2(\W_{H^U})$.  In this
concluding section, we will show that this family admits a modification
which is jointly continuous in $(u,t)$.

Let $\varphi ,\psi \in\scr S$ and $u,v\in\R^2$ be given.
Then, by Proposition \ref{OU}, for $\W_{H^U}$-almost every $\omega \in\U$,
$$\iint\limits_{\R^2} \varphi(s) \psi(t) (S^u(\omega)) (s) (S^v(\omega)) (t)\,dsdt = 
{}_{\scr S}\langle\varphi ,\omega * \mu^{*u}\rangle_{\scr S'}{}_{\scr
  S}\langle\psi ,\omega * \mu^{*v}\rangle_{\scr S'},$$ 
where the integral in the left-hand side is absolutely convergent.
Because $\E_{\W_{H^U}}\bigl[S^u\omega (t)^2\bigr]$ is finite and independent of
$(u,t)\in\R^2$, by taking the expectation with respect to $\mathcal{W}_{H^U}$ and
using \eqref{alpha}, one can pass from this to
\begin{align*}
&\iint\limits_{\R^2} \varphi(s) \psi(t) \E_{\mathcal{W}_{H^U}}\bigl[(S^u(\omega)) (s) (S^v(\omega)) (t)\bigr]
\,dsdt =\E_{\mathcal{W}_{H^U}}\Bigl[ 
{}_{\scr S}\langle\varphi ,\omega * \mu^{*u}\rangle_{\scr S'}{}_{\scr
  S}\langle\psi ,\omega * \mu^{*v}\rangle_{\scr S'}\Bigr]\\&=
\frac{2}{\pi} \int_{-\infty}^{\infty} 
\frac{e^{2i (u-v) \operatorname{Arctg} (2\lambda)}}{1+4\lambda^2} \, \widehat{\varphi} (\lambda)
 \, \overline{\widehat{\psi}(\lambda)} d \lambda  
= \frac{2}{\pi}
 \iiint\limits_{\R^3}\frac{e^{i  [(t-s)\lambda + 2 (u-v)
       \operatorname{Arctg} (2\lambda)]}}{1+4\lambda^2} \,\varphi(s) \psi(t)
 \,ds dt d \lambda.\end{align*}
Hence, 
\begin{equation}\label{cov}
\mathbb{E}_{\mathcal{W}_{H^U}} [S^u(\omega)) (s) (S^v(\omega)) (t) ]  = 
\frac{2}{\pi} \int_{-\infty}^{\infty} 
\frac{e^{i  [(t-s)\lambda + 2 (u-v) \operatorname{Arctg}
      (2\lambda)]}}{1+4\lambda^2} \, d \lambda,
\end{equation}
first for almost every and then, by continuity,
for all $(s,t) \in \mathbb{R}^2$.
In particular, we now know that the $\W_{H^U}$-distribution of 
$\{S^u(\omega))(t):\,(u,t)\in\R^2\}$ is stationary.

To show that there is a continuous version of this process, we will use
Kolmogorov's continuity criterion, which, because it is stationary and
gaussian, comes down to showing that 
$$\left| 1 - \mathbb{E}_{\W_{H^U}}[(S^u(\omega))(s)(S^v(\omega))(t)] \right| 
\le C\bigl|(u,s)-(v,t)\bigr|^\alpha $$
for some $C < \infty$ and $\alpha >0$.  But
\begin{align*}
&\left| 1 - \mathbb{E}_{\W_{H^U}}[(S^u(\omega))(s)(S^v(\omega))(t)] \right| 
\leq \frac{2}{\pi} \, \int_{-\infty}^{\infty} \frac{d \lambda}{1 + 4 \lambda^2} 
\, \left|e^{i [(t-s)\lambda +  2(u-v) \operatorname{Arctg} (2\lambda)]} - 1 \right|
\\ & \leq \frac{2}{\pi} \, \int_{-\infty}^{\infty} \frac{d \lambda}{1 + 4 \lambda^2} 
\, \left|e^{i (t-s)\lambda } - 1 \right|  + \frac{2}{\pi} \, \int_{-\infty}^{\infty}
 \frac{d \lambda}{1 + 4 \lambda^2} 
\, \left|e^{ 2i(u-v) \operatorname{Arctg} (2\lambda)} - 1 \right|
\\&\leq \frac{2}{\pi} \, \int_{-\infty}^{\infty}  \frac{d \lambda}{1 + 4 \lambda^2} 
(|t-s| |\lambda| \wedge 2)  + \frac{4}{\pi} \, \int_{-\infty}^{\infty}  
 \frac{d \lambda}{1 + 4 \lambda^2} | (u-v) \operatorname{Arctg} (2 \lambda)  |,
\end{align*}
and, after simple estimation, this shows that
$$
\left| 1 - \mathbb{E}[(S^u(\omega))(s)(S^v(\omega))(t)] \right| 
\leq C \left[ |u-v| + |t-s| \left(1 + \log \left( 1 + \frac{1}{(t-s)^2}
  \right) \right) \right], 
$$
where $C<\infty $.  Clearly, the desired conclusion follows.
\newline
\begin{remark}
A question about filtrations comes naturally when one considers the 
group of transformations $(S^u)_{u \in \mathbb{R}}$ on the space $\mathcal{U}$. 
Indeed, for all $t, u \in \mathbb{R}$, let $\mathcal{F}_{t}^u$ be the $\sigma$-algebra generated by 
the $\mathcal{W}_{H^U}$-negligible subsets of $\mathcal{U}$ of and the variables $(S^u(\omega))(s)$, 
for $s \in (-\infty,t]$ (these variables are well-defined up to a negligible set). From the results of Jeulin and Yor, one quite easily deduces the following properties of the filtrations of the 
form $(\mathcal{F}_t^u)_{t \in \mathbb{R}}$ for $u \in \mathbb{R}$: 
\begin{itemize}
\item For all $t, u \in \mathbb{R}$, $\mathcal{F}_t^{u}$ is generated by $\mathcal{F}_t^{u+1}$ and $(S^u(\omega))(t)$.
\item For all $t, u \in \mathbb{R}$, $\mathcal{F}_t^{u+1}$ and $(S^u(\omega))(t)$ are independent under $\mathcal{W}_{H^U}$.
\item For all $t, u \in \mathbb{R}$, the decreasing intersection of $\mathcal{F}_t^{u+n}$ for $n \in \mathbb{Z}$ is trivial (i.e. it satisfies 
the zero-one law).
\item If $u \in \mathbb{R}$ is fixed, the $\sigma$-algebra generated by $\mathcal{F}_t^{u+n}$ for $t \in \mathbb{R}$ does not depend
on $n \in \mathbb{Z}$.
\end{itemize}
\noindent
All these statements concern the sequence of filtrations $(\mathcal{F}^{u+n})_{n \in \mathbb{Z}}$ for fixed $u \in \mathbb{R}$. 
A natural question arises:  how can these results be extended to the continuous family of filtrations $(\mathcal{F}^{u})_{u \in \mathbb{R}}$? 
Unfortunately, for the moment, we have no answer to this question (in particular the family does not seem to be decreasing with $u$). 
\end{remark} 
\medbreak


\begin{thebibliography}{1}

\bibitem{JY}
T. Jeulin, M. Yor, \emph{Filtration des ponts browniens et \'equations
  diff\'erentielles stochastiques lin\'eaires}, S\'eminaire de
{P}robabilit\'es, {XXIV}, Lecture Notes in Math., 
  vol. 1426, Springer-Verlag, 1990, pp.~227--265.

\bibitem{M}
 P.-A. Meyer, \emph{Sur une transformation du mouvement brownien due \`a
   Jeulin et Yor}, S\'eminaire de {P}robabilit\'es, {XXVIII}, Lecture Notes
 in Math., 
  vol. 1583, Springer-Verlag, 1994, pp.~98--101.
 
 \bibitem{L}
 N.-N. Lebedev, \emph{Special functions and their applications} (Translated
 from Russian), Dover, New York, 1972. 

  
\bibitem{Y}
M. Yor, \emph{Some aspects of Brownian motion. Part I: Some Special
  Functionals}, Lectures in Mathematics, Ed. ETH Z\"urich, Birkha\"user,
1992. 

\bibitem{St1} D. Stroock, \emph{Probability Theory, an Analytic View, 2nd edition},
  Cambridge University Press, 2011.

\bibitem{St2} D. Stroock, \emph{Some thoughts about Segal's ergodic theorem},
  Colloq.\ Math.\ vol.\ 118 \#1, pp. 89-105, 2010.
\end{thebibliography}
\end{document}